\documentclass[preprint,12pt]{elsarticle}
\usepackage{amsfonts,amssymb, amsmath, amsthm}
\usepackage{fullpage}
\usepackage{graphicx}
\usepackage{float}
\restylefloat{figure}

\newtheorem{thm}{Theorem}[section]
\newtheorem{lemma}[thm]{Lemma}
\newtheorem{corollary}[thm]{Corollary}

\numberwithin{equation}{section}

\journal{Physics Letters A}

\newcommand{\R}{\mathbb{R}}

\begin{document}

%\linenumbers

\begin{frontmatter}

\title{Stabilization with target oriented control for higher order difference equations}

\author[label1]{E. Braverman}
\author[label2]{D. Franco}
\address[label1]{Department of Mathematics and Statistics, University of
Calgary, 2500 University Drive N.W., Calgary, AB T2N 1N4, Canada}
\address[label2]{ Departamento de Matem\'atica Aplicada, E.T.S.I. 
Industriales, Universidad Nacional de Educaci\'on a Distancia~(UNED), c/~Juan del Rosal~12, 28040, Madrid, Spain}

\begin{abstract}
For a physical or biological model whose dynamics is described by a higher order difference equation
$u_{n+1}=f(u_n,u_{n-1}, \dots, u_{n-k+1})$, we propose a version of a target oriented control
$u_{n+1}=cT+(1-c)f(u_n,u_{n-1}, \dots, u_{n-k+1})$, with $T\ge 0$, $c\in [0,1)$.
In ecological systems, the method incorporates harvesting and recruitment and for a wide class
of $f$, allows to stabilize (locally or globally) a fixed point of $f$. If
a point which is not a fixed point of $f$ has to be stabilized, the target oriented control
is an appropriate method for achieving this goal. As a particular case, we consider  
pest control applied to pest populations with delayed density-dependence. This corresponds to
a proportional feedback method, which includes harvesting only, for higher order equations.
\end{abstract}

%\pacs{05.45.-a~~Nonlinear dynamics and chaos}


\end{frontmatter}

%\classification{05.45.-a~~Nonlinear dynamics and chaos}

%%%%%%%%%%%%%%%%%%%%%%%%%%%%%%%%%%%%%%%%%%%%
%% MAINMATTER
%%%%%%%%%%%%%%%%%%%%%%%%%%%%%%%%%%%%%%%%%%%%

\section{Introduction}

Controlling chaos consists in stabilizing nonlinear systems with chaotic dynamics 
\cite{andrievskii2003control,corron2000controlling,sanjuan2010recent}. 
The target state for control can be a steady state, a periodic orbit, or even a particular aperiodic trajectory in a chaotic attractor.
The best way to achieve any of these goals depends on the nature of the situation modeled by the system.
For instance, in population dynamics, control methods incorporating either harvesting or stocking, or both, are more appropriate 
than others, e.g. OGY method \cite{ott1990controlling}, which performs continuous small perturbations in the parameters of the 
system. 

Using OGY method, stabilization of otherwise unstable fixed point or cycle was achieved experimentally in \cite{ditto}.
The experimental system consisted of a gravitationally buckled, amorphous magnetoelastic ribbon.
The chosen ribbon material exhibited very large reversible changes of Young's modulus with the application of small
magnetic fields. Oscillation of the ribbon were brought to a chosen regime using the control of the distance from
the chosen orbit \cite{ditto}. However, the stabilized orbit had to be an orbit of the unperturbed system.

The main motivation of our investigation is population dynamics where chaotic orbits coming close enough to zero 
can threaten population survival and cause unpredictable changes in a relevant food chain. Higher density values 
may be as dangerous as low ones due to either extinction stipulated by the previous overpopulation and accumulated pollution 
\cite{ricker1954stock}, or connected to paradox of enrichment \cite{rosenzweig}. If the target is pest eradication
then harvesting at each time step can be an appropriate strategy. However, if the purpose is to keep the population in certain
bounds, the control should incorporate both harvesting and stock recruitment. A recently suggested target oriented control
\cite{Dattani} seems to be an adequate method to achieve this goal: the control intensity, as in OGY method, depends on
the deviation of the stock size from the chosen target---the more distant the population density from the target value, the more 
intensive the control is. Since 2011, when the target oriented control was introduced, there were several developments 
\cite{Chaos2014,TPC} justifying a possibility to stabilize a fixed point (for a prescribed target) and exploring 
the form of the stock-recruitment dependency when such stabilization is possible. The purpose of the present
paper is to overcome the following shortcomings of previous contributions:
\begin{enumerate}      
\item
Most of control methods considered suitable for stabilizing population dynamics have been studied in the framework of first order difference equations  \cite{braverman2012stabilization, Dattani, franco2014stabilizing, franco2013stabilization, gm,
hilker2013harvesting, liz2010control, schreiber2001chaos,tung2014comparison}. In any physical application, this means that we consider a scalar system only, and this system has no memory: the next state depends on the previous state only.
In biological applications, first order equations can describe non-structured populations that survive for one season only,
with overwintering offspring. The framework of higher order equations allows to consider multi-seasonal interactions even with non-overlapping 
generations. Moreover, dynamics of a structured population in certain cases can be described by a higher order difference equation; e.g. \cite{franco2014global}.  

\item
OGY and several similar methods (for example, prediction based control \cite{braverman2012stabilization}) 
are focused on stabilization of either a fixed point or an unstable orbit
of the original system. Other strategies, such as proportional feedback \cite{gm}, 
involve harvesting only and stabilize a point which is typically closer 
to zero than fixed points of the non-controlled equation. However, physical or ecological considerations (for example, required density
to sustain functioning of a food chain or providing a sufficient supply for harvesting)
lead to the necessity to stabilize a point which {\em is not} a fixed point 
of the system. Here we illustrate that, with target oriented control, it is possible.
It will allow to keep a controlled population at a prescribed level, certainly, at a cost.
\end{enumerate}

The present paper is devoted to a general control strategy with the goal 
of stabilizing steady states of a broad class of higher order difference equations which include chaotic and non-chaotic systems.

We recall that first order difference equations 
\begin{equation}
\label{ec_dim1}
u_{n+1}=h(u_n), \quad n= 0,1,2,\dots
\end{equation}     
are suitable for modeling single species populations with non-overlapping generations. The map $h$ usually takes the form 
$h(x)=x g(x)$ with $g$ being a positive decreasing map regulating the intraspecific competition for resources. 
The Ricker model \cite{ricker1954stock}, $u_{n+1}=u_n \exp(r-u_n)$, and the Beverton-Holt model \cite{beverton1957dynamics}, 
$u_{n+1}= r  u_n / (1+u_{n})$, are examples of \eqref{ec_dim1} broadly used in theoretical ecology. 
However, the evidence of the existence of explicit time lags in the intraspecific regulatory mechanisms for some 
species \cite{levin1976note} makes it necessary to incorporate delays in the equations to obtain more realistic 
models. Examples of such equations related to the Ricker and Beverton-Holt models are the delay Ricker equation \cite{levin1976note}
\begin{equation}
\label{Rickerdelayed}
u_{n+1}=u_n \exp(r-u_{n-k+1})
\end{equation}
and the Pielou equation \cite{Pielou1}
\begin{equation}
\label{Pielou}
u_{n+1}=\frac{r u_n}{1+u_{n-k+1}},
\end{equation}
where $k\ge 2$ is a fixed natural number determining the time lag. Clearly, both \eqref{Rickerdelayed} and \eqref{Pielou} are 
special cases of the $k$th-order difference equation  
\begin{equation}
\label{original}
u_{n+1}=f(u_n,u_{n-1}, \dots, u_{n-k+1}), \quad n=0,1,2,\dots 
\end{equation}
where $f\colon \R^k_+ \to \R_+$, $\R_+=[0,\infty)$. Moreover, similarly to the first order equation \eqref{ec_dim1}, the higher 
order difference equation \eqref{original}, 
and particularly \eqref{Rickerdelayed} and \eqref{Pielou},  can show complicated dynamics. Therefore, the design and study of control strategies for the higher order equation \eqref{original} 
is a natural extension of the stabilization problem for the first order equation \eqref{ec_dim1}. 

Here, we generalize a method called target oriented control. Target oriented control (TOC) method 
\begin{equation}
\label{TOC_eq}
u_{n+1}=h\left( c T+(1-c)u_n \right), \quad T\ge 0, \; c\in [0,1),
\end{equation}
was developed in \cite{Dattani} with the aim to stabilize the dynamics of the first order equation \eqref{ec_dim1}. Parameter $T$ is called target 
and parameter $c$ measures the control intensity. Essentially, TOC increases the state variable if it is smaller than the target and reduces it if it is larger. 
For the values of $c$ close enough to one, TOC can provide 
global stabilization of unimodal maps with a negative Schwarzian derivative \cite{TPC} and some other models, where the smoothness conditions are 
relaxed  \cite{Chaos2014}. 
Following \cite{TPC}, we note that \eqref{TOC_eq} is a combination of the linear transformation of the variable
 \begin{equation}
 \label{phidef}
 \phi(x)=cT+(1-c)x
 \end{equation}
and the function $h$. Moreover, if we  switch the order and consider the modified target oriented control (MTOC) 
 \begin{equation}
 \label{MTOC_eq}
 u_{n+1}=cT+(1-c)h(u_n), \quad T\ge 0, \; c\in [0,1),
 \end{equation}
 then the fixed point $K_c$ of the controlled equation \eqref{TOC_eq} is globally (locally) asymptotically stable 
 if and only if the fixed point $P_c=\phi(K_c)$ of \eqref{MTOC_eq} is globally (locally) asymptotically stable.

 The framework of MTOC provides a natural generalization of target control methods to higher order difference equations.
 In particular, we propose the following control applied to the uncontrolled equation \eqref{original} 
 \begin{equation}
 \label{MTOC_eq_high}
 u_{n+1}=cT+(1-c)f(u_n,u_{n-1}, \dots, u_{n-k+1}), \quad T\ge 0, \; c\in [0,1).
 \end{equation}
 Further we will refer to the controlled equation \eqref{MTOC_eq_high} as HMTOC (higher-order modified target oriented control).

If we assume the zero target $T=0$ in \eqref{TOC_eq}, then we obtain the proportional feedback method (PF),
\begin{equation}
\label{PF_eq}
u_{n+1}=h((1-c) u_n), \quad c \in [0,1),
\end{equation}
consisting in a reduction of the state variable, proportional to the size of this variable \cite{gm}. 
The assumption of the proportional reduction is aligned with the idea of constant effort harvesting, without any stocking.  
In \eqref{PF_eq}, harvesting occurs before reproduction.
Switching the variable reduction function $\psi(x)=(1-c)x$ with the map $h$, we get a modified proportional feedback method (MPF) 
in which harvesting takes place after reproduction 
\begin{equation}
\label{MPF_eq} 
u_{n+1}=(1-c) h(u_n), \quad c \in [0,1).
\end{equation} 
Similarly, the control in \eqref{MPF_eq} can be extended to involve higher order methods. We obtain a modified version 
of the proportional feedback control for higher order equations
\begin{equation} 
\label{MPF_eq_high}
u_{n+1}=(1-c)f(u_n,u_{n-1}, \dots, u_{n-k+1}), \quad c \in [0,1),
\end{equation} 
which is a particular case of HMTOC for the target $T=0$.

The main results of the present paper are the following:
\begin{enumerate}
\item 
We obtain sufficient conditions for local and global stabilization of a fixed point with 
HMTOC. As illustrated by numerical examples, some of these conditions are sharp.
\item 
Stabilization of the zero equilibrium with the proportional feedback method is considered
as a particular case of the above results.
\item 
The minimal stabilizing control intensity $c$ is estimated. The flexibility for the choice
of the point to be stabilized is one of the advantages of HMTOC. However, stabilization conditions
depend on this choice.  
\end{enumerate}

The paper is organized as follows. Section 2 deals with local stabilization of higher order equations. In Section 3, we justify the possibility of 
global stabilization. Section 4 contains the proof of the fact that with target oriented control, any prescribed point can be stabilized. 
Some examples and numerical illustrations are presented in Section 5.
Finally, Section 6 contains a summary.

%%%%%%%%%%%%%%%%%%%%%%%%%%%%%%%%%%%%
\section{Local stabilization }

Our first result gives a sufficient condition for HMTOC to have at least a positive fixed point for all control intensities. It complements and slightly generalizes Lemma 1 in \cite{TPC}. All the proofs of results in this section can be found in the Appendix.
\begin{lemma} 
\label{lemma_exist}
(i) Assume that $f\colon \R^k_+ \to \R_+$ is continuous and there exists a positive constant $M$ such that $f(M,\dots,M) \le  M$. Then, for any $T \in (0,M]$ there exists at least a positive fixed point $P_c$ of HMTOC in $(0,M]$  for every $c \in (0, 1)$. 

(ii) In particular, if there exists a positive constant $\overline{M}$ such that $f(x,\dots,x) \le  x$ for $x\ge \overline{M}$, then for 
any fixed target $T>0$ there exists at least a positive fixed point $P_c$ of HMTOC in the interval $(0,\max\{T,\overline{M} \}]$  for 
every $c \in (0, 1)$.
\end{lemma}

Next, we show that HMTOC is able to asymptotically stabilize a fixed point if a sufficiently strong control is implemented.

\begin{thm}
\label{p_esta}
Assume that $f\colon \R^k_+ \to \R_+$ is continuously differentiable and that there exists a bounded interval $I\subset \R_+$ such that HMTOC with target $T>0$ has at least a fixed point $P_c$ in $I$ for every $c \in (0, 1)$. Then, there exists $c^*\ge 0$ such that $P_c$ is asymptotically stable for $c\in  (c^*, 1)$.   
\end{thm}

From a practical perspective, it is convenient to have an estimation as good as possible for $c^*$ in the statement of Theorem~\ref{p_esta}. Following its proof, this essentially reduces to prove that a polynomial is a Schur polynomial (i.e. with all roots smaller than one in module).  Although there are characterizations for a Schur polynomial, as for example the Schur test or the Jury conditions \cite[Corollaries 3.4.91 and 3.4.98]{hinrichsen2005mathematical},  
obtaining an explicit expression for $c^*$ is difficult for two main reasons. First, because if $k$ is large, 
then expressions in these characterizations get too complicated. Second, because we need to know with precision the interval $I$ where the fixed point $P_c$ of HMTOC is, in order to reduce the range of possible values for the partial derivatives of $f$. Of course, the second issue 
disappears if the target $T=K$ is a fixed point of $f$, that is,  
\begin{equation}
\label{fixed}
K=f(K,K, \dots,K),
\end{equation}
because then $K$ is also a fixed point of HMTOC for every $c \in (0, 1)$,  and the interval $I$ reduces to the point $K$. 
In the next result, we take advantage of this fact and present a sharp estimation of the control intensity necessary to stabilize 
a fixed point $K$ of $f$ when $k=2$ and the target coincides with the fixed point $T=K$.

\begin{thm}
\label{p_esta_k2}
Assume that $f\colon \R^2_+ \to \R_+$ is continuously differentiable and $K\in \R_+$  satisfies $K=f(K,K)$. Then, the fixed point 
$K$ of the controlled equation HMTOC with the target $T=K$ is asymptotically stable for $c\in (c^*,1)$ where
\[
c^* = \max\left \{0, 1-\frac{1}{\max\{|\frac{\partial f}{\partial y}(\mathbf{K})| ,|\frac{\partial f}{\partial x}(\mathbf{K})|+\frac{\partial f}{\partial y}(\mathbf{K})\}} \right  \},
\] 
($c^*=0$ if $\frac{\partial f}{\partial x}(\mathbf{K})=\frac{\partial f}{\partial y}(\mathbf{K})=0$), and $\mathbf{K}=(K,K)$.
\end{thm}

As we have said, if $k$ is larger, then optimal expressions for $c^*$ get complicated. 
Nevertheless, when \eqref{fixed} holds, it is possible to get easy-to-calculate estimates such as the following one.
\begin{thm}
\label{p_esta_sufficient}
Assume that $f\colon \R^k_+ \to \R_+$ is continuously differentiable and $K\in \R_+ $ satisfies $K=f(K,K, \dots, K)$. Then, the 
fixed point $K$ of the controlled equation HMTOC with the target $T=K$ is asymptotically stable for $c\in (c^*,1)$, where
\[
c^* = \max\left \{0, 1-\frac{1}{\sum_{j=1}^k |\frac{\partial f}{\partial x_j}(\mathbf{K})|} \right  \},
\]  
$\mathbf{K}=(K,K, \dots, K)$.   
\end{thm}

\section{Global stabilization of a fixed point}

In this section, we assume that the point $K\in \R_+$ to be stabilized is a fixed point for the uncontrolled system, that is, it satisfies \eqref{fixed}. Further, we assume that there exists $L>0$ such that
\begin{equation} 
\label{Lipschitz_2}
| f(\mathbf{x})-K | \leq L \| \mathbf{x}-(K,K, \dots, K) \|, ~~\mathbf{x} \in \R^k_+
\end{equation} 
where $\mathbf{x}=(x_1, \dots,x_k)$ and $\| \mathbf{x} \| =\max_{1\leq j\leq n} |x_j|$,
however the proofs below are easily adapted to any other vector norm. We note that inequality \eqref{Lipschitz_2} 
implies that $K$ is a fixed point of $f$.

Our first result gives not only a sufficient condition for global stability of a fixed point of the controlled 
equation but also an estimate of the control intensity necessary to reach it, which depends on the constant $L$ in 
condition \eqref{Lipschitz_2}.  Its proof uses some ideas from \cite{BBL_JDEA2005}. \begin{thm}
\label{theorem_add3}
Assume that $f\colon \R^k_+ \to \R_+$ satisfies condition \eqref{Lipschitz_2}. 
Then for $c \in (c^*,1)$ with $c^*= \max \{0,1-\frac{1}{L}\}$, the fixed point $K$ is a globally 
asymptotically stable fixed point for the controlled equation HMTOC with $T=K$, that is, any sequence 
starting with $(u_{1-k},u_{2-k},\dots,u_{0}) \in \R^k_+$ and satisfying HMTOC with $T=K$ converges to $K$.
\end{thm}
\begin{proof} 
Without loss of generality, let us assume that $L\ge 1$ in inequality~\eqref{Lipschitz_2}. We fix $\theta \in (0,1)$ and 
let $c=1-\frac{\theta}{L}$, where $c\in (0,1)$.
Then, using equation \eqref{MTOC_eq_high} and inequality~\eqref{Lipschitz_2}, we have
\begin{eqnarray*}
|u_{n+1}-K| & = & (1-c)|f(u_n,u_{n-1}, \dots, u_{n-k+1})-K|
\nonumber \\ &=& \frac{\theta}{L} |f(u_n,u_{n-1}, \dots, 
u_{n-k+1})-K| 
\nonumber 
\\ & \leq & \theta \frac{L}{L} 
\| (u_n-K,u_{n-1}-K, \dots, u_{n-k+1}-K) \| 
\nonumber \\
&=&  \theta \max_{n-k+1 \leq j \leq n} |u_j-K|,
\end{eqnarray*}
therefore
\begin{equation*}
\label{contraction}
| u_{n+1}-K| \leq \theta \max_{n-k+1 \leq j \leq n} |u_j-K|.
\end{equation*}
Continuing this process, we obtain
\begin{equation}
\label{induction}
|u_{n+l}-K| \leq \theta \max_{n-k+1 \leq j \leq n} |u_j-K|,~~l=1, \dots, k.
\end{equation}
Using \eqref{induction} as an induction step, we obtain
\begin{equation*}
\label{estimate}
| u_{n}-K| \leq \theta^{[n/k]} \max_{-k+1\leq j \leq 0} |u_j-K|,
\end{equation*}
where $[t]$ is the integer part of $t$. Then 
\[
\lim_{n\to\infty} u_n=K,
\]
moreover, there is a guaranteed rate of convergence. Thus, $K$ is a globally asymptotically stable 
fixed point of equation~\eqref{MTOC_eq_high} for any $c \in (c^*,1)$ with 
\begin{equation*}
\label{alpha}
c^*=1-\frac{1}{L}.
\end{equation*}

\end{proof}

Applying Theorem~\ref{theorem_add3} to \eqref{MPF_eq_high}, we obtain the following 
result about the stabilization of the origin with proportional control.

\begin{corollary}
	\label{corollary_add3}
	Assume that $f\colon \R^k_+ \to \R_+$ satisfies condition \eqref{Lipschitz_2} with $K=0$. Then for $c \in 
(c^*,1)$ with $c^*= \max \{0,1-\frac{1}{L}\}$  zero is a globally asymptotically stable fixed point for  the controlled equation \eqref{MPF_eq_high}.
\end{corollary}

In Theorem~\ref{theorem_add3} and Corollary~\ref{corollary_add3}, constant $L$ in \eqref{Lipschitz_2} is used to estimate the 
control intensity necessary to stabilize globally a fixed point: the smaller the value of $L$ the sooner the global stability is 
attained. In some cases the calculation of $L$ can be direct,  as for the Pielou equation with $r\ge 1$ where 
\[
f(\mathbf{x})= f(x_1,\dots,x_k)= \frac{r x_1}{1+x_k}
\]
satisfies \eqref{Lipschitz_2} with $K=r-1$ and $L=r$:
\begin{align*}
|f(\mathbf{x})-K| &= \left| \frac{r x_1}{1+x_k} - (r-1) \right|  =  \left| \frac{r x_1 - (r -1) x_k-(r-1)}{1+x_k} \right| \\
& =\left| \frac{r(x_1-(r-1))}{1+x_k} - \frac{(r-1)(x_k-(r-1))}{1+x_k} \right| \\
&\leq r |x_1-(r-1)|  + (r-1) |x_k-(r-1)| \leq L \| \mathbf{x}-(K,\dots,K)\|, \quad \mathbf{x} \in \R^k_+.
\end{align*}
However,  in general finding $L$ could be hard. Therefore, it is interesting to have easy ways to calculate $L$ for a given map. Evidently, if $f$ is globally Lipschitz continuous and $K$ is a fixed point of $f$, then we could take $L$ as the global Lipschitz constant of $f$. 
The proof of the next result shows how to calculate $L$ if $f$ is a locally Lipschitz continuous bounded function.

\begin{lemma}
Let $f\colon \R^k_+ \to \R_+$ be a locally Lipschitz continuous bounded function and $K$ be a fixed point of $f$, then there exists 
$L\ge 1$ such that condition~\eqref{Lipschitz_2} holds for $\mathbf{x} \in \R^k_+$.
\end{lemma}
\begin{proof} Since $f$ is bounded, there exists $A>0$ such that
\begin{equation}   
\label{bound}
0\leq f(\mathbf{x}) \leq A, ~~\mathbf{x} \in \R^k_+. 
\end{equation} 

Moreover, the local Lipschitz continuity of $f$ guarantees that condition~\eqref{Lipschitz_2}
is satisfied for $\mathbf{x}=(x_1,x_2,\dots,x_k)$  such that $0 \leq x_i \leq 2K$,  $i=1, 2, \dots, k$ with some constant $\tilde{L}$.

Next, let at least one of $x_i$ satisfy $x_i>2K$, then 
\begin{equation}
\label{bound_1}
\| \mathbf{x} - (K,K, \dots, K) \|>K.
\end{equation}
Due to \eqref{bound} and $A\geq K>0$, we have
\[
|f(\mathbf{x})-K| \leq \max\{ A-K, K \},
\]
which together with \eqref{bound_1} gives
\begin{equation}
\label{bound_2}
|f(\mathbf{x})-K| \leq \max\left\{ \frac{A}{K}-1, 1 \right\} \| \mathbf{x}-(K,K, \dots, K) \|
\end{equation}
for $\mathbf{x} \not\in [0,2K]^k$.
Thus, we obtain that
\begin{equation*}
\label{Lipschitz_3}
| f(\mathbf{x})-K| \leq L \| \mathbf{x}-(K,K, \dots, K) \|, ~~\mathbf{x} \in \R^k_+
\end{equation*} 
holds with $\displaystyle L=\max\left\{ \tilde{L}, \frac{A}{K}-1, 1 \right\}$. 
\end{proof}

%Many functions employed in applications have the zero as a fixed point and positive fixed points. We note that previous result establish that the control intensity   

%%%%%%%%%%%%%%%%%%%%%%%%%%%%%%%%%%%%%%%
\section{Stabilization of an arbitrary point}

The results contained in previous section allow us to consider
stabilization of a fixed point only. If we aim to stabilize a different point with HMTOC, then we need to apply
a ``corrective" HMTOC first. The next result shows that after applying both such a corrective step and a stabilizing control, we 
are still using HMTOC.

\begin{lemma}
	\label{lemma_add1}
	A combination of two HMTOCs is a HMTOC.
\end{lemma}
\begin{proof} 
If $\phi_1(x)=c_1T_1+(1-c_1)x$ is applied after another argument transformation 
$\phi_2(x)=c_2T_2+(1-c_2)x$, then 
\begin{eqnarray*}
	\phi_1\left(\phi_2(f(x_1,x_{2}, \dots, x_{k})) \right) & = & c_1T_1+(1-c_1) \left[c_2T_2+(1-c_2)f(x_1,x_{2}, 
\dots, x_{k})\right]
	\\ &=& c_1T_1+(1-c_1)c_2T_2+(1-c_1)(1-c_2)f(x_1,x_{2}, \dots, x_{k})
	\\ &=& c_3T_3+(1-c_3)f(x_1,x_{2}, \dots, x_{k}),
\end{eqnarray*}
where 
\[ 
c_3=c_1+c_2-c_1c_2, \quad \mbox{ and } \quad T_3=\frac{c_1T_1+c_2T_2-c_1 c_2 T_2}{c_3}.
\]
Obviously $\alpha := (1-c_1)(1-c_2) \in (0,1)$ as long as $c_1,c_2 \in (0,1)$,
so $c_3=1-\alpha \in (0,1)$. On the other hand, the value of $T_3$ is positive as a ratio of two 
positive numbers. \end{proof}

The following result was justified in \cite[Lemma 5]{Chaos2014}.
\begin{lemma}
	\label{lemma_add2a}
	Let $f_1 \colon \R_+ \to \R_+$ be a continuous function satisfying $f_1(x)>0$ for $x>0$. Then for any $K>0$ in the range of $f$ there exist $c_K\in (0,1)$ and $T_K\geq 0$ such that $K$ is a fixed point of $g(x)=c_KT_K+(1-c_K)f_1(x)$.
\end{lemma}

Lemma~\ref{lemma_add2a} immediately implies the following result.

\begin{corollary}
	\label{lemma_add2}
	Let $f \colon \R^k_+ \to \R_+$ be a continuous function satisfying $f(x,x,\dots,x)>0$ for $x>0$. Then for any  $K\in \{ f(x,x,\dots,x): x >0\}$ there exist $c_K\in (0,1)$ and $T_K\geq 0$ such that $K$ is a fixed point of $g(\textbf{x})=c_K T_K+(1-c_K)f(\textbf{x})$.
\end{corollary}
\begin{proof} 
Apply Lemma~\ref{lemma_add2a} with $f_1(x)=f(x,x,\dots,x)$. \end{proof}

Let us demonstrate that for suitable functions any point can be 
stabilized with a combination of HMTOCs.

\begin{thm}
\label{theorem_add_1}

Let $f\colon \R^k_+ \to \R_+$ be a continuous differentiable function with $f(x,x,\dots,x)>0$ for $x>0$. 
Then for any $K_1\in \{ f(x,x,\dots,x): x >0\}$ there exists a combination of two HMTOCs for which 
$K_1$ is an asymptotically stable fixed point.

Additionally, let either  $f\colon \R^k_+ \to \R_+$ be globally Lipschitz continuous, or let~\eqref{Lipschitz_2} 
hold for a fixed point $K$, or $f$ be a globally bounded function. 
Then for any $K_1\in \{ f(x,x,\dots,x): x >0\}$ there exists a combination of two HMTOCs for which 
$K_1$ is a globally asymptotically stable fixed point.
\end{thm}

\begin{proof} 
By Corollary~\ref{lemma_add2} there exists $c_{K_1}\in (0,1)$ and $T_{K_1}\geq 0$ such that $K_1$ is a fixed point of 
$g(\textbf{x})=c_{K_1} T_{K_1}+(1-c_{K_1})f(\textbf{x})$. 

Therefore, we can apply Theorem~\ref{p_esta} 
to guarantee that there exists $c^*\in (0,1)$ such that for $c\in (c^*,1)$ the combination of HMTOCs 
\begin{equation}
\label{comp}
u_{n+1}=c K_1 + (1-c)[c_{K_1} T_{K_1}+(1-c_{K_1})f( u_n,u_{n-1}, \dots, u_{n-k+1} )]
\end{equation}
has $K_1$ as an asymptotically stable fixed point.

Additionally, let either  $f\colon \R^k_+ \to \R_+$ be globally Lipschitz continuous, or condition~\eqref{Lipschitz_2} hold for a fixed point $K$,
or $f$ be a globally bounded function. We note that the conditions of the theorem imply that $f$ is locally Lipschitz. From now on, we write $\mathbf{K}=(K,K, \dots, K)$ and $\mathbf{K_1}=(K_1,K_1, \dots, K_1)$.
Let us assume that after the first HMTOC, the function 
$g({\bf x})=c_{K_1} T_{K_1}+(1-c_{K_1}) f({\bf x})$ satisfies $g(\mathbf{K_1}) = K_1$, therefore 
$K_1 - c_{K_1} T_{K_1} =(1-c_{K_1}) f(\mathbf{K_1})$. Thus
\begin{align*}
\left| g({\bf x}) - K_1 \right|  &= \left| c_{K_1} T_{K_1}+(1-c_{K_1}) f({\bf x})-K_1 \right| \\
&=\left|  (1-c_{K_1})[f({\bf x})-f(\mathbf{K_1})] \right| \\
& \leq L (1-c_{K_1}) \| {\bf x}-\mathbf{K_1} \|, 
~{\bf x} \in \R^k_+,
\end{align*}
if $f$ is globally Lipschitz with the constant $L$. In consequence, inequality~\eqref{Lipschitz_2} holds for $g$ and $K_1$.

Next, consider the case when $f$ satisfies~\eqref{Lipschitz_2} for a fixed point $K$. By the local Lipschitz 
condition,  it is possible to choose $L_2>0$ such that 
\begin{align}
\label{local}
\left| g({\bf x}) - K_1 \right| &= (1-c_{K_1}) \left| f({\bf x})-f(\mathbf{K_1}) \right| \notag \\
 & \leq (1-c_{K_1}) L_2  \| {\bf x}-\mathbf{K_1} \| \quad \mbox{for} \quad  \| {\bf x}-\mathbf{K_1}\| 
\leq \max\{K,K_1\},
~{\bf x} \in \R^k_+.
\end{align}
Moreover, for any ${\bf x} \in \R^k_+$ such that $\| {\bf x}-\mathbf{K_1}\| \geq \max\{K,K_1\}$, we have 
$\| {\bf x}-\mathbf{K_1} \| \geq K_1=\| \mathbf{K_1} \|$ and
$\| {\bf x}-\mathbf{K_1} \| \geq K=\| \mathbf{K} \|$. Hence
\begin{align*}
\left| g({\bf x}) - K_1 \right| 
&= (1-c_{K_1})\left| f({\bf x})-f(\mathbf{K_1}) \right|\\ 
& \leq (1-c_{K_1})\left| f({\bf x})-K \right|+ (1-c_{K_1})\left| f(\mathbf{K_1})-K \right|\\
&\leq L (1-c_{K_1}) \| {\bf x}-\mathbf{K} \| + L (1-c_{K_1})\| \mathbf{K_1}- \mathbf{K}  \| \\
 & = L (1-c_{K_1})\| {\bf x}-\mathbf{K_1}+\mathbf{K_1}-\mathbf{K} \| + L (1-c_{K_1}) 
\left[ \|  \mathbf{K}\| + \|\mathbf{K_1} \| \right] \\
& \leq L (1-c_{K_1})\| {\bf x}-\mathbf{K_1} \| + 2 L (1-c_{K_1}) \left[ \|  \mathbf{K}\| + \|\mathbf{K_1} \| \right]\\
& \leq L (1-c_{K_1}) \| {\bf x}-\mathbf{K_1} \| + 4 L (1-c_{K_1})  \| {\bf x}-\mathbf{K_1} \|. 
\end{align*}
Thus 
\[
\left| g({\bf x})-K_1 \right| \leq L_1 \| {\bf x}-\mathbf{K_1} \|,
~{\bf x} \in \R^k_+
\]
where $L_1=(1-c_{K_1})\max\{L_2, 5L\}$, and $L_2$, $L$ appear in \eqref{local} and \eqref{Lipschitz_2}, respectively.

Finally, a similar argument proves that inequality~\eqref{Lipschitz_2} holds for $g$ and $K_1$ when $f$ is globally bounded.

Further, Theorem~\ref{theorem_add3} guarantees that there exists $c^*\in (0,1)$ such that for $c\in (c^*,1)$ the combination of HMTOCs has $K_1$ as a globally asymptotically stable fixed point.  
\end{proof}

%%%%%%%%%%%%%%%%%%%%%%%%%%%%%
\section{Examples and numerical simulations}

In this section, three examples  
illustrate how the previous results can be used to determine a control intensity 
sufficient to stabilize a steady state. 

\subsection{Stabilization of a nontrivial fixed point}
Let us consider the second order equation 
\begin{equation}
\label{ecexp}
u_{n+1}=\exp(1-u_n) \exp(1 -u^2_{n-1}),
\end{equation}
which has a stable 3-cycle and an unstable fixed point $K=1$. 
Assume that we are interested in stabilizing such a fixed point. To estimate the control intensity for that goal, we begin by 
calculating the partial derivatives of the map $f(x,y)= \exp(1-x) \exp(1 -y^2)$, 
\[
  \frac{\partial f}{\partial x} (x,y) = -\exp(1-x) \exp(1 -y^2),  \qquad \frac{\partial f}{\partial y} (x,y)= -2 y \exp(1-x) \exp(1 -y^2). 
\]
Evaluating them at $\mathbf{K}=(1,1)$ and using Theorem~\ref{p_esta_k2},  $K=1$ is a locally stable fixed point for the controlled equation HMTOC with 
the target $T=1$ if $c\in (c^*,1)$ with 
\[
c^*= 1-\frac{1}{\max\{|-2|,|-1|-2\}}= 1-\frac{1}{2}=0.5.
\]   

Moreover, the partial derivatives of $f$ satisfy 
\[
\left | \frac{\partial f}{\partial x} (x,y) \right | = |-\exp(1-x) \exp(1 -y^2) |\le e^2 \approx 7.39, \quad (x,y)\in \mathbb R^2_+,
\]
and
\[ 
\left  |\frac{\partial f}{\partial y} (x,y)\right |= |-2 y \exp(1-x) \exp(1 -y^2) |\le 2 e \frac{\sqrt{e}}{\sqrt{2}} \approx 6.34, \quad (x,y)\in \mathbb R^2_+,
\] 
where we have used that, in $\mathbb R_+$, the function $\exp(1-x)$ is decreasing and the function $x\exp(1-x^2)$ is bounded by $\sqrt{e}/\sqrt{2}$. 

In consequence, by the mean value theorem in several variables and the Cauchy-Schwarz inequality, map $f$ is globally Lipschitz continuous with constant $L=2 e^2 \approx 14.78$ and  condition~\eqref{Lipschitz_2} holds for the same $L$. Using Theorem~\ref{theorem_add3} we obtain that the fixed point $K=1$ is a globally asymptotically stable fixed point of the controlled equation HMTOC for at least any $c$ greater than $1-\frac{1}{2e^2}\approx 0.93$.   

% L can be improved to  $L=\sqrt{2} (\sqrt{e^4+2e^3}) \approx 13.76$, but I am not sure if it is worth....

\begin{figure}[h]
	\centering
	\includegraphics[width=0.75 \linewidth]{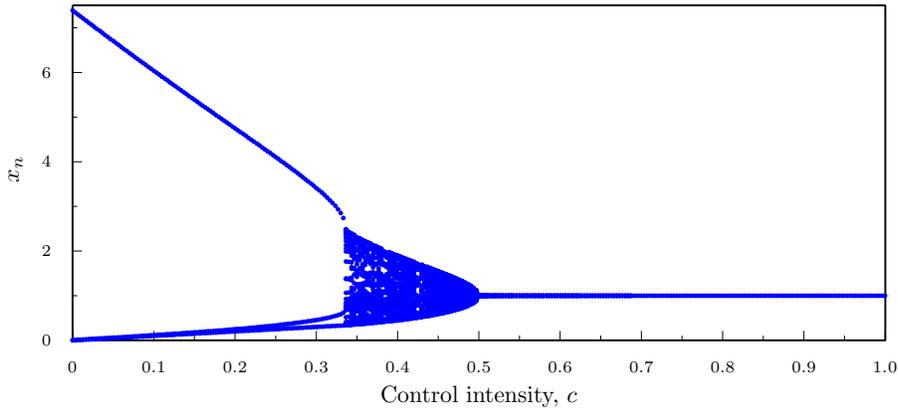}
	\caption{Stabilizing effect of HMTOC on the uncontrolled equation $u_{n+1}=\exp(1-u_n) \exp(1 -u^2_{n-1})$. Target was chosen as $T=1$, which is a fixed point of the uncontrolled system. For each $c\in k/300$, $k=1,\dots,300$, we plotted 50 consecutive values of $u_n$ after discarding the first 3000. Initial conditions were chosen pseudo-randomly.}
	\label{fig:FiguraBifBounded}
\end{figure}

Figure~\ref{fig:FiguraBifBounded} shows the effect of increasing the control intensity $c$ from $0$ to $1$ in equation~\eqref{ecexp}. We observe that the interval of control intensities guaranteeing local stability of the fixed point is sharp. However, numerically it seems that the interval of control intensities guaranteeing the global attraction derived from Theorem~\ref{theorem_add3} could be improved. 
We emphasize that such an improvement could be obtained directly from Theorem~\ref{theorem_add3}, if we show that condition~\eqref{Lipschitz_2} holds for a smaller $L$. 
Next example illustrates that using the best constant $L$ in condition~\eqref{Lipschitz_2} can give sharp results for global attraction.   %On the other hand, Figure~\ref{fig:FiguraBifBounded} illustrates that increasing the control intensity can apparently increase the complexity of the dynamics  before the fixed point gets stabilized. 

\subsection{Pest eradication. Stabilization of the trivial fixed point.}

Let us consider the delayed Ricker equation
$u_{n+1}=u_n \exp(1.5 -u_{n-1})$, which has been proposed as a model for populations with non-overlapping generations and a one generation time lag in the intraspecific competition for resources \cite{levin1976note}. Suppose that we want to stabilize the fixed point $K=0$. For example, this could be the case if the population is a pest and we are interested in eradicating it. We choose the target $T=0$ in HMTOC, which biologically corresponds to harvesting a constant proportion of the population after reproduction. 

Then, analogously to the previous example, Theorem~\ref{p_esta_k2} implies that $K=0$ is a locally stable fixed point for the controlled equation \eqref{MPF_eq_high} if $c\in (c^*,1)$ with 
\[
c^*=1-\frac{1}{\max\{|0|,|e^{1.5}|+0\}}=1-\frac{1}{e^{1.5}}\approx 0.78.
\]   
Or in other words, that any constant harvesting effort greater than $c^*$ is able to control the pest if the initial population is small enough.       

In order to calculate the harvesting effort to control the pest independently of the initial population size, we note that condition~\eqref{Lipschitz_2} holds with $L= e^{1.5}$ and $K=0$ because
\[
|x \exp(1.5-y)|\le e^{1.5} |x|\le  e^{1.5} \max\{|x|,|y|\}, \quad  (x,y)\in \mathbb R^2_+.
\] 
Thus, Corollary~\ref{corollary_add3} implies that $K=0$ is indeed a globally asymptotically stable fixed point for $c$ greater than $1-\frac{1}{e^{1.5}}\approx 0.78$. Figure~\ref{fig:FiguraBif} illustrates this example. Note how only for $c$ greater than $0.78$ the trivial fixed point is stabilized. Therefore, the estimates obtained from Theorem~\ref{p_esta_k2} and Corollary~\ref{corollary_add3} cannot be improved in this case.

\begin{figure}[h!]
\centering
\includegraphics[width=0.75 \linewidth]{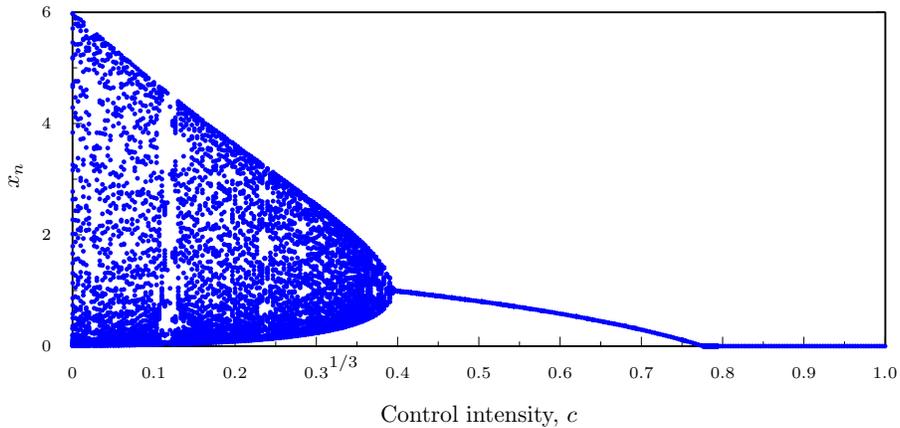}
\caption{Stabilizing effect of increasing the control parameter $c$ in the controlled \eqref{MPF_eq_high}. The uncontrolled equation is $u_{n+1}=u_n \exp(1.5 -u_{n-1})$. For each $c\in k/300$, $k=1,\dots,300$, we plotted 50 consecutive values of $u_n$ after discarding the first 3000. Initial conditions were chosen pseudo-randomly. }
\label{fig:FiguraBif}
\end{figure}

\subsection{Targeting. Stabilization of an arbitrary value.}
The aim of this example is to show that using a combination of HMTOCs is very flexible from the point of view of targeting, that is, a suitable combination of HMTOCs allows us to carry the system into a desired objective. Let us consider the third order delayed Pielou equation
\[
u_{n+1}=\frac{8 u_n}{1+u_{n-2}},
\]
which was proposed as a discrete analogue of the  logistic differential equation with delay \cite{Pielou1}. 

To illustrate that a combination of two HMTOCs can stabilize any point $K \in \{ \frac{8 x}{1+x}: x >0\}=[0,8)$, let 
us choose, for example, $K=6$, which is not a fixed point of the uncontrolled equation. 
It is easy to verify that taking $c_K=2/9$ and $T_K=3$ solves the equation
\[
K= c_K T_K+ (1-c_K) \frac{8\cdot K}{1+K}.
\]

Therefore, by Theorem~\ref{theorem_add_1} we know that point $K=6$ is stabilized when using alternatively  HMTOC with $c=c_K, T=T_K$ and with $T=K$ and $c\in (0,1)$ if $c$ is large enough in latter.  Figure~\ref{fig:FiguraPielou} numerically illustrates it.  

\begin{figure}[h]
	\centering
	\includegraphics[width=0.75 \linewidth]{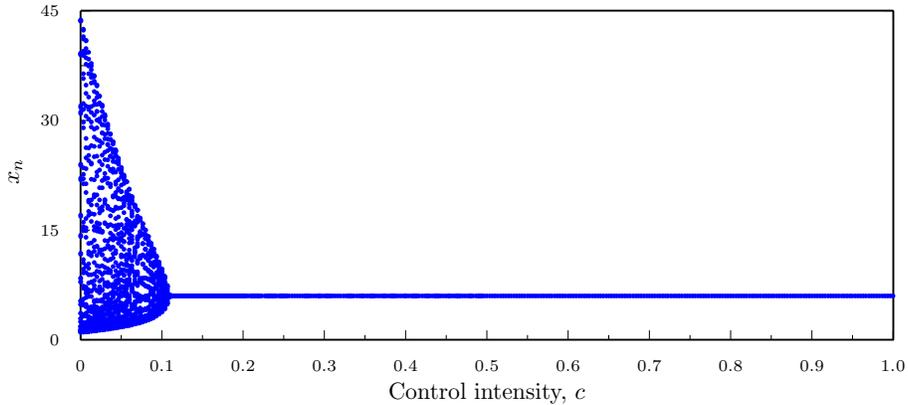}
	\caption{Stabilizing effect of a combination of two HMTOCs for the third order Pielou equation $u_{n+1}=\frac{8 u_n}{1+u_{n-2}}$. The combination was chosen to guarantee the stabilization of $K=6$. For each $c\in k/300$, $k=1,\dots,300$, we plotted 50 consecutive values of $u_n$ after discarding the first 3000. Initial conditions were chosen pseudo-randomly.}
		\label{fig:FiguraPielou}
	\end{figure}

\section{Summary}

Considering stabilization of higher order equations by a natural generalization of target oriented control \cite{Dattani}, we have obtained:
\begin{enumerate}
\item 
sufficient local stabilization results; % which, as illustrated by numerical simulations, are close to being necessary;
\item 
sufficient global stabilization conditions;
\item 
estimates of the control intensity necessary to achieve stabilization.  
\end{enumerate}

In the second order case, the estimate of the minimum control intensity needed to locally stabilize a fixed point 
given in Theorem~\ref{p_esta_k2} is sharp as the first and second examples illustrate. Moreover, the second example shows that, at least in some cases, Corollary~\ref{corollary_add3} is optimal. On the other hand, the numerical simulations suggest that our main global stability result (Theorem~\ref{theorem_add3}) 
could be improved.  
As remarked, one way to attain this improvement will be to obtain 
the smallest constant $L$ satisfying condition~(3.1). 
Another one could be to study a restricted family of maps, as in \cite{TPC} where the results hold for unimodal maps with a negative Schwarzian 
derivative. But probably this improvement will be difficult. This expected difficulty should not be surprising if we recall that the formulated in 1976 conjecture about local stability implying global stability for the delayed Ricker equation \cite{levin1976note} was only recently solved for the second order case using a computer aided proof \cite{{bartha2013local}}.

Most of stabilization results in the present paper were developed in the case when 
a fixed point of the original difference equation is stabilized.
Later on, we describe how the target oriented control can shift 
a fixed point to any prescribed value. In this context,
the scheme how all the previous stabilization results can be applied is the following: 
\begin{itemize}
\item
We define $c_1$ and $T_1$ so that the prescribed value $x^{\ast}$ is a fixed point of the equation
$g(\textbf{x})=c_{1} T_{1}+(1-c_{1})f(\textbf{x})$.
\item
For the new higher order equation $\textbf{x}_{n+1}=g(\textbf{x}_n)$, the chosen $x^{\ast}$ is a fixed point,
so any of previous results on either local or global stabilization apply, with $T_2=x^{\ast}$.
In particular, these results allow to find bounds for $c_2 \in (c^*,1)$ leading to stabilization. The process follows the proof of Theorem~\ref{theorem_add_1}.  
\item
By Lemma~\ref{lemma_add1}, we combine two successive HMTOCs getting
$c_3=c_1+c_2(1-c_1)$, where $c_1$ is fixed, $c_2 \in (c^*,1)$, so $c_3 \in (c_1+c^*(1-c_1), 1)$,
and $T_3=(c_1T_1+c_2 x^{\ast}(1-c_1))/c_3$ is uniquely defined once $c_3$ is chosen.
Thus, we know the resulting $c_3$ and $T_3$ required for stabilization of $x^{\ast}$.
\end{itemize}

In population dynamics, when the population has age or spatial structure, 
this leads to systems of difference equations. In some relevant cases the 
system can be rewritten as a delayed equation (see, for example, \cite{liz2012global}, where a juvenile-adult population is considered) to which HMTOC 
could be applied.  
The method presented in this paper does not apply in the general system setting. 
However, the extension to systems should be straightforward from the results given 
here and will be discussed elsewhere.

\centerline{\bf Acknowledgment}

The authors would like to thank the editor and four anonymous referees for their constructive comments.  
The first author was partially supported by NSERC, grant RGPIN/261351-2010. 
The second author was supported by the Spanish Ministerio de Econom\'{\i}a 
y Competitividad and FEDER, grant MTM2013-43404-P.
Part of this work was accomplished while D. Franco was visiting Isaac Newton Institute 
for Mathematical Sciences in Cambridge; he thanks the institute for support and hospitality.

\appendix

\section{}

\noindent \emph{Proof of Lemma~\ref{lemma_exist}.}
To prove claim (i) we note that fixed points of HMTOC are solutions of
\begin{equation}
\label{eq_fp}
x = g(x) := c T + (1-c) f(x,\dots,x),
\end{equation}
where $g$ is a continuous function.
We have 
\[ g(0)=c T + (1-c) f(0,\dots,0) \ge c T > 0,\]
and
\[
g(M)= c T + (1-c) f(M,\dots,M) \le c T + (1-c)M \le M.
\]
Thus, equation \eqref{eq_fp} has at least one positive solution $x \in (0,T)$  by the intermediate value theorem.  

Claim (ii) follows from claim (i) by taking $M=\max\{T,\overline{M} \}$.   \qed\\

\noindent \emph{Proof of Theorem~\ref{p_esta}}.
It is well known that a fixed point  $P_c \in I$ of HMTOC is asymptotically stable if $\{z\in \mathbb{C} : |z|<1\}$ contains all the roots of the characteristic polynomial	
\begin{equation*}
	p_{P_c}(x ):= x^k - \frac{\partial f}{\partial x_1 }(\mathbf{P}_c) \ x^{k-1} - \dots -  \frac{\partial f}{\partial x_k }(\mathbf{P}_c)	
\end{equation*}
where $\mathbf{P}_c=(P_c,\dots,P_c)$; see \cite[Section 1.2]{camouzis2007dynamics} for more details. 
	
Applying Fujiwara's upper bound (see for instance \cite[Remark 8]{deutsch1970matricial}) to polynomial  $p_{P_c}$,  we obtain that any root $\lambda$ of $p_{P_c}$ satisfies
\[
|\lambda | \le 2 (1-c) \max_{i=1,\dots,k} \left | \frac{\partial f}{\partial x_i }(\mathbf{P}_c) \right |^{\frac{1}{i}}.
\]
	
Next, since the partial derivatives of $f$ are bounded on the rectangle $I^k$, there exists a positive real number $A$ such that 
	\begin{equation}
	\label{boundc}
	A=\max_{P_c \in I} \left \{ \max_{i=1,\dots,k} \left | \frac{\partial f}{\partial x_i }(\mathbf{P}_c) \right |^{\frac{1}{i}} \right \}.	\end{equation} 
	  Therefore,  taking 
	\begin{equation*}
	c^* = \max\left\{ 0,1-\frac{1}{2 A} \right\},
	\end{equation*}
	we obtain that all the roots $\lambda$ of the polynomial $p_{P_c}$ satisfy $|\lambda | <1$ for $c\in (c^*,1)$. \qed\\

\noindent \emph{Proof of Theorem~\ref{p_esta_k2}.}
 Following the reasoning of Theorem~\ref{p_esta}, we have to guarantee that all the roots of the characteristic 
polynomial
\[
x^2-(1-c)\frac{\partial f}{\partial x}(\mathbf{K}) x - (1-c)\frac{\partial f}{\partial y}(\mathbf{K}),
\] 
are smaller than one in modulus. Applying the Jury conditions \cite{hinrichsen2005mathematical,jury1971inners}, we obtain that a necessary and sufficient condition for that to happen is 
\[
2>1-(1-c) \frac{\partial f}{\partial y}(\mathbf{K}) > (1-c)\left |\frac{\partial f}{\partial x}(\mathbf{K}) \right |,
\]
which is equivalent to 
\[ 
\left |\frac{\partial f}{\partial y}(\mathbf{K})\right |< \frac{1}{1-c} \qquad \mbox{ and } \qquad \left  |\frac{\partial f}{\partial x}(\mathbf{K}) \right |+\frac{\partial f}{\partial y}(\mathbf{K})< \frac{1}{1-c},
\]
from which we obtain the desired result. \qed\\

\noindent \emph{Proof of Theorem~\ref{p_esta_sufficient}}.
The result follows from \cite[Theorem 1.2.5]{camouzis2007dynamics}, after noting that the conditions of the theorem imply that 
\begin{equation*}
\label{c_bound}
\sum_{j=1}^k (1-c) \left| \frac{\partial f}{\partial x_j}(\mathbf{K}) \right| <1.
\end{equation*}
 \qed

\end{document}